\documentclass[a4paper,twoside]{article}
\usepackage{a4}
\usepackage{amssymb}
\usepackage{amsmath}
\usepackage{upref}
\usepackage[active]{srcltx}
\usepackage[pagebackref,colorlinks,citecolor=blue,linkcolor=blue]{hyperref}
\newcount\minutes \newcount\hours
\hours=\time
\divide\hours 60
\minutes=\hours
\multiply\minutes -60
\advance\minutes \time
\newcommand{\klockan}{\the\hours:{\ifnum\minutes<10 0\fi}\the\minutes}
\newcommand{\tid}{\today\ \klockan}
\newcommand{\prtid}{\smash{\raise 10mm \hbox{\LaTeX ed \tid}}}
\renewcommand{\prtid}{}
\makeatletter
\def\sectionmark#1{} 
\def\subsectionmark#1{}
\newcommand{\sectnr}{\ifnum \c@secnumdepth >\z@
                 \thesection.\hskip 1em\relax \fi}
\def\@evenhead{\footnotesize\rm\thepage\hfil\leftmark\hfil\llap{\prtid}}
\def\@oddhead{\footnotesize\rm\rlap{\prtid}\hfil\rightmark\hfil\thepage}
\def\tableofcontents{\section*{Contents} 
 \@starttoc{toc}}
\makeatother
\makeatletter
\def\@biblabel#1{#1.}
\makeatother
\makeatletter
\let\Thebibliography=\thebibliography
\renewcommand{\thebibliography}[1]{\def\@mkboth##1##2{}\Thebibliography{#1}
\addcontentsline{toc}{section}{References}
\frenchspacing 
\setlength{\@topsep}{0pt}
\setlength{\itemsep}{0pt}%
\setlength{\parskip}{0pt plus 2pt}%
}
\makeatother
\makeatletter
\def\mdots@{\mathinner.\nonscript\!.%
 \ifx\next,.\else\ifx\next;.\else\ifx\next..\else
 \nonscript\!\mathinner.\fi\fi\fi}
\let\ldots\mdots@
\let\cdots\mdots@
\let\dotso\mdots@
\let\dotsb\mdots@
\let\dotsm\mdots@
\let\dotsc\mdots@
\def\vdots{\vbox{\baselineskip2.8\p@ \lineskiplimit\z@
    \kern6\p@\hbox{.}\hbox{.}\hbox{.}\kern3\p@}}
\def\ddots{\mathinner{\mkern1mu\raise8.6\p@\vbox{\kern7\p@\hbox{.}}%
    \raise5.8\p@\hbox{.}\raise3\p@\hbox{.}\mkern1mu}}
\makeatother
\makeatletter
\let\Enumerate=\enumerate
\renewcommand{\enumerate}{\Enumerate%
\setlength{\@topsep}{0pt}
\setlength{\itemsep}{0pt}%
\setlength{\parskip}{0pt plus 1pt}%
\renewcommand{\theenumi}{\textup{(\alph{enumi})}}%
\renewcommand{\labelenumi}{\theenumi}%
}
\let\endEnumerate=\endenumerate
\renewcommand{\endenumerate}{\endEnumerate\unskip}
\makeatother
\makeatletter
\def\@seccntformat#1{\csname the#1\endcsname.\quad}
\makeatother
\newcommand{\authortitle}[2]{\author{#1}\title{#2}\markboth{#1}{#2}}
\newcommand{\auth}[2]{{#1, #2.}}
\newcommand{\art}[6]{{\sc #1, \rm #2, \it #3\/ \bf #4 \rm (#5), \mbox{#6}.}}

\newcommand{\artprep}[3]{{\sc #1, \rm #2, \rm #3.}}

\newcommand{\arttoappear}[3]{{\sc #1, \rm #2, to appear in \it #3}}
\newcommand{\book}[3]{{\sc #1, \it #2, \rm #3.}}
\newcommand{\AND}{{\rm and }}
\RequirePackage{amsthm}
\newtheoremstyle{descriptive}%
  {\topsep}   
  {\topsep}   
  {\rmfamily} 
  {}          
  {\bfseries} 
  {.}         
  { }         
  {}          
\newtheoremstyle{propositional}%
  {\topsep}   
  {\topsep}   
  {\itshape}  
  {}          
  {\bfseries} 
  {.}         
  { }         
  {}          
\newtheoremstyle{remarkstyle}%
  {\topsep}   
  {\topsep}   
  {\rmfamily}  
  {}          
  {\itshape} 
  {.}         
  { }         
  {}          
\theoremstyle{propositional}
\newtheorem{thm}{Theorem}[section]
\newtheorem{prop}[thm]{Proposition}
\newtheorem{lem}[thm]{Lemma}
\newtheorem{cor}[thm]{Corollary}
\theoremstyle{descriptive}
\newtheorem{deff}[thm]{Definition}
\newtheorem{example}[thm]{Example}
\makeatletter
\renewenvironment{proof}[1][\proofname]{\par
  \pushQED{\qed}%
  \normalfont
  \trivlist
  \item[\hskip\labelsep
        \itshape
    #1\@addpunct{.}]\ignorespaces
}{%
  \popQED\endtrivlist\@endpefalse
}
\makeatother
\newcommand{\setm}{\setminus}

\def\vint{\mathop{\mathchoice%
          {\setbox0\hbox{$\displaystyle\intop$}\kern 0.22\wd0%
           \vcenter{\hrule width 0.6\wd0}\kern -0.82\wd0}%
          {\setbox0\hbox{$\textstyle\intop$}\kern 0.2\wd0%
           \vcenter{\hrule width 0.6\wd0}\kern -0.8\wd0}%
          {\setbox0\hbox{$\scriptstyle\intop$}\kern 0.2\wd0%
           \vcenter{\hrule width 0.6\wd0}\kern -0.8\wd0}%
          {\setbox0\hbox{$\scriptscriptstyle\intop$}\kern 0.2\wd0%
           \vcenter{\hrule width 0.6\wd0}\kern -0.8\wd0}}%
          \mathopen{}\int}
\DeclareMathOperator{\Lip}{Lip}
\DeclareMathOperator{\interior}{int}
\DeclareMathOperator*{\essinf}{ess\,inf}
\newcommand{\loc}{_{\rm loc}}
\newcommand{\alp}{\alpha}
\newcommand{\ga}{\gamma}
\newcommand{\de}{\delta}
\newcommand{\eps}{\varepsilon}
\newcommand{\la}{\lambda}
\newcommand{\p}{{$p\mspace{1mu}$}}
\newcommand{\R}{\mathbf{R}}
\newcommand{\Sphere}{\mathbf{S}}
\newcommand{\Z}{\mathbf{Z}}
\newcommand{\Np}{N^{1,p}}
\newcommand{\Nploc}{N^{1,p}\loc}
\newcommand{\ut}{\tilde{u}}
\newcommand{\wh}{\widehat{w}}
\newcommand{\muh}{\hat{\mu}}
\newcommand{\mut}{\tilde{\mu}}
\newcommand{\Ga}{\Gamma}
\newcommand{\Lploc}{L^p\loc}
\newcommand{\uhat}{{\hat{u}}}
\newcommand{\wt}{\widetilde{w}}
\newcommand{\CPI}{C_{\rm PI}}
\numberwithin{equation}{section}
\newcommand{\eqv}{\mathchoice{\quad \Longleftrightarrow \quad}{\Leftrightarrow}
                {\Leftrightarrow}{\Leftrightarrow}}
\newcommand{\imp}{\mathchoice{\quad \Longrightarrow \quad}{\Rightarrow}
                {\Rightarrow}{\Rightarrow}}
\newenvironment{ack}{\medskip{\it Acknowledgement.}}{}

\begin{document}

\authortitle{Anders Bj\"orn, Jana Bj\"orn
    and Nageswari Shanmugalingam}
{Locally \p-admissible measures on $\R$}
\author{
Anders Bj\"orn \\
\it\small Department of Mathematics, Link\"oping University, \\
\it\small SE-581 83 Link\"oping, Sweden\/{\rm ;}
\it \small anders.bjorn@liu.se
\\
\\
Jana Bj\"orn \\
\it\small Department of Mathematics, Link\"oping University, \\
\it\small SE-581 83 Link\"oping, Sweden\/{\rm ;}
\it \small jana.bjorn@liu.se
\\
\\
Nageswari Shanmugalingam
\\
\it \small  Department of Mathematical Sciences, University of Cincinnati, \\
\it \small  P.O.\ Box 210025, Cincinnati, OH 45221-0025, U.S.A.\/{\rm ;}
\it \small  shanmun@uc.edu
}

\date{}
\maketitle

\noindent{\small
{\bf Abstract}.
In this note we show that locally \p-admissible measures on $\R$ 
necessarily come from local Muckenhoupt $A_p$ weights.
In the proof we employ the corresponding characterization
of global \p-admissible measures on $\R$
in terms of global $A_p$ weights due to Bj\"orn, Buckley and Keith,
together with tools from analysis in metric spaces, more specifically 
preservation of the doubling condition and Poincar\'e inequalities under 
flattening, due to Durand-Cartagena and Li.

As a consequence, the class of locally \p-admissible weights on $\R$ 
is invariant under addition and satisfies the lattice property.
We also show that measures that are \p-admissible on an interval
can be partially extended by periodical reflections to global \p-admissible
measures.
Surprisingly, the \p-admissibility has to hold on a larger interval than 
the reflected one, and an example shows that this is necessary. 

}

\bigskip


\medskip


\section{Introduction}

Globally \p-admissible weights for Sobolev spaces 
and differential equations on $\R^n$ 
were introduced in 
Heinonen--Kilpel\"ainen--Martio~\cite{HeKiMa}.
Four conditions were imposed on such weights, 
which were later reduced to the following two conditions (the remaining 
two being redundant),
see  \cite[2nd ed., Section~20]{HeKiMa}.
Even earlier, such weights were used to study regularity of 
linear degenerate
elliptic equations (with $p=2$) in Fabes--Jerison--Kenig~\cite{FaJeKe} and 
Fabes--Kenig--Serapioni~\cite{FaKeSe}.

\begin{deff}
A measure $\mu$ on $\R^n$ is \emph{globally \p-admissible}, $1 \le p <\infty$,
if it is globally doubling 
and supports a global \p-Poincar\'e inequality.
If $d\mu=w\,dx$ we also say that 
the weight $w$ is \emph{globally \p-admissible}.
\end{deff}

Muckenhoupt $A_p$ weights are globally \p-admissible
(see  \cite[Theorem~15.21]{HeKiMa} and \cite[Theorem~4]{JBFennAnn}), 
but the converse
is not true in $\R^n$, $n \ge 2$.
On the other hand, on $\R$ even globally \p-admissible \emph{measures}
are given by global $A_p$ weights,
as was shown in Bj\"orn--Buckley--Keith~\cite[Theorem~2]{BBK-Ap}.

In many situations it is local, rather than global, properties
that are relevant, especially when dealing with local properties
such as regularity of solutions to differential equations, 
see the studies in 
Danielli--Garofalo--Marola~\cite{DaGaMa}, Garofalo--Marola~\cite{GaMa} 
and Holopainen--Shanmugalingam~\cite{HoSh}.
There are several different possibilities for formulating
local doubling conditions and local  Poincar\'e inequalities. 
The conditions we impose on the measure 
do not
require any uniformity in the constants nor in the radii involved, 
and are thus truly local.
However, uniformity is natural in many situations, and
then we are able to conclude slightly more, see
Section~\ref{sect-uniform}.

The principal aim of this paper is to obtain the following characterization
of locally \p-admissible measures on $\R$.

\begin{thm} \label{thm-p-adm-char-intro}
Let $p\ge1$ and let $\mu$ be a measure on $\R$.
Then the following are equivalent\/\textup{:}
\begin{enumerate}
\renewcommand{\theenumi}{\textup{(\roman{enumi})}}%
\item \label{d-loc}
$\mu$ is locally \p-admissible\/\textup{;}
\item \label{d-loc-Ap}
$d\mu=w\,dx$, where $w$ is a local $A_p$ weight\/\textup{;}
\item \label{d-glob-Ap}
$d\mu=w\,dx$, and for each bounded interval $I\subset \R$ there is 
a global $A_p$ weight $\wt$ on $\R$ such that $\wt=w$ on $I$.
\end{enumerate}
\end{thm}

As a consequence of these characterizations we obtain the lattice
property for locally \p-admissible weights on $\R$, as well as the 
preservation of local \p-admissibility when taking finite sums of measures,
see Section~\ref{sect-cor-Ap}.
This complements some results in Kilpel\"ainen--Koskela--Masaoka~\cite{KilKoMa},
where such questions were studied for global $A_p$ weights 
and globally \p-admissible measures on $\R^n$.
As a byproduct, we provide an elementary proof of 
\cite[Proposition~4.3]{KilKoMa}, see Lemma~\ref{lem-Ap-lattice}.

This note is  a continuation of the systematic development of
local and semilocal doubling measures and Poincar\'e
inequalities on metric spaces from 
Bj\"orn--Bj\"orn~\cite{BBsemilocal} and~\cite{BBnoncomp}.
Local assumptions are also
natural for studying \p-harmonic and quasiharmonic
functions, and Theorem~\ref{thm-p-adm-char-intro} plays a role in 
Liouville type theorems on the real line, see
Bj\"orn--Bj\"orn--Shanmugalingam~\cite{BBSliouville}.

In \cite{ChuaWheeden}, Chua and Wheeden extensively studied which 
types of Poincar\'e inequalities hold on intervals,
and also obtained optimal constants.
Using their results, we show that,
in contrast to Theorem~\ref{thm-p-adm-char-intro},
a \p-admissible weight on a bounded interval 
is \emph{not} necessarily an $A_p$ weight on 
that interval, see Example~\ref{ex-not-extension}
and also Theorem~\ref{thm-local-Ap}.

The proof of
Theorem~\ref{thm-p-adm-char-intro} turned
out to be more complicated than we had expected,
and considerably more involved than 
the proof of the corresponding global result in
Bj\"orn--Buckley--Keith~\cite[Theorem~2]{BBK-Ap}.
In addition to careful estimates, we also use the metric space
theory.
More specifically, to show that a locally \p-admissible measure $\mu$
is absolutely continuous, 
we create a suitable \p-admissible measure  
on the circle $\Sphere^1$, and
then use a flattening argument
due to Durand-Cartagena--Li~\cite{Esti-Xining}
to obtain a globally $q$-admissible measure $\muh$
on $\R$ for some $q$.
We then have at our disposal the global result in 
\cite[Theorem~2]{BBK-Ap}
which yields
that $\muh$ is absolutely continuous and that the 
corresponding weight $\wh$ (given by $d\muh=\wh\,dx$) is
an $A_q$ weight.
The number $q$ obtained from \cite{Esti-Xining}
can be larger than $p$ 
(and depends on  $\mu$ as well as on the interval used in constructing
$\muh$). 
However, the only consequence we need from this step is that $\muh$
is absolutely continuous and hence so is $\mu$.
Once the absolute continuity of $\mu$ is in place we 
instead use a direct argument to
deduce the local
$A_p$ condition. To complete
the proof we also need
the fact from Bj\"orn--Bj\"orn~\cite{BBsemilocal}
that locally \p-admissible measures are semilocally \p-admissible.

Having characterized the locally \p-admissible measures $\mu$  on $\R$,
it is also  interesting to know how the minimal \p-weak upper gradients
behave for functions in the Newtonian Sobolev space $\Np(\R,\mu)$.
If $u$ is locally absolutely continuous on some interval,
then the fundamental
theorem of calculus shows that $|u'|$ is an upper gradient for $u$, and thus
$g_u\le|u'|$ a.e. 
For Lipschitz functions $u$ and arbitrary measures on $\R$, 
the minimal \p-weak upper gradient $g_u$ has been fully described in 
Di~Marino--Speight~\cite[Theorem~2]{DiMar-Spe}.
The following result addresses this question for general Newtonian functions
and weights on $\R$.

\begin{prop} \label{prop-ac}
Let $\mu$ be a measure on $\R$
and $1 < p <\infty$.
Assume that $d\mu = w \, dx$ and $w,w^{1/(1-p)}\in L^1\loc(I)$ for 
some\/ \textup{(}not necessarily open\/\textup{)}
interval $I\subset\R$.
Then every $u\in \Nploc(I,\mu)$ is locally absolutely continuous on $I$
and $g_u=|u'|$ a.e.
\end{prop}

In particular, the proposition applies if $w$ is locally \p-admissible 
on $\R$
(and thus a local $A_p$ weight, by Theorem~\ref{thm-p-adm-char-intro})
with $p>1$.
Note that $\Nploc(I,\mu)$ is a refinement of the standard
Sobolev space $W^{1,p}\loc(I,w)$ 
considered in
Heinonen--Kilpel\"ainen--Martio~\cite{HeKiMa}, 
see the discussion at the end of Section~\ref{sect-prelim}.

\begin{ack}
The first two authors were supported by the Swedish Research Council
grants 2016-03424 and 621-2014-3974, respectively. The third
author was partially supported by the NSF grant DMS-1500440. 
During part of 2017--18 the third author was a guest 
professor at Link\"oping University,
partially
funded by the Knut and Alice Wallenberg Foundation;
she
thanks them for their kind support and hospitality.
\end{ack}

\section{Metric spaces}
\label{sect-prelim}

We are primarily interested in measures and weights on $\R$, but
we will also need to use tools from first-order analysis on
metric spaces.
In this section we discuss the definitions used in metric spaces.
For proofs of the facts stated in this section 
we refer the reader to Bj\"orn--Bj\"orn~\cite{BBbook} and
Heinonen--Koskela--Shanmugalingam--Tyson~\cite{HKST}.

We assume throughout the paper that $1 \le p<\infty$ 
and that $X=(X,d,\mu)$ is a metric space equipped
with a metric $d$ and a positive complete  Borel  measure $\mu$ 
such that $0<\mu(B)<\infty$ for all balls $B \subset X$.
We assume throughout the paper that balls are open.
We let $B=B(x,r)=\{y \in X : d(x,y) < r\}$ denote the ball
with centre $x$ and radius $r>0$, and let $\la B=B(x,\la r)$.
In metric spaces it can happen that
balls with different centres and/or radii 
denote the same set;
we will however adopt the convention that a ball $B$ comes with
a predetermined centre $x_B$ and radius $r_B$. 

We primarily deal with $X$ being the real line $\R$, and in this case 
balls and bounded open intervals are the same objects. 
We will use both nomenclatures and notations.

We follow Heinonen and Koskela~\cite{HeKo98} in introducing
upper gradients as follows (in~\cite{HeKo98} they are referred to 
as very weak gradients).
A \emph{curve} is a continuous mapping from an interval.
We will only consider curves which are nonconstant, compact and
rectifiable,
i.e.\ of finite length.
A curve can thus be parameterized by its arc length $ds$.

\begin{deff} \label{deff-ug}
A nonnegative Borel function $g$ on $X$ is an \emph{upper gradient} 
of an extended real-valued function $u$
on $X$ if for all 
curves  
$\gamma: [0,l_{\gamma}] \to X$,
\begin{equation} \label{ug-cond}
        |u(\gamma(0)) - u(\gamma(l_{\gamma}))| \le \int_{\gamma} g\,ds,
\end{equation}
where we follow the convention that the left-hand side is considered to be $\infty$ 
whenever at least one of the terms therein is $\pm \infty$.
If $g$ is a nonnegative measurable function on $X$
and if (\ref{ug-cond}) holds for \p-almost every curve (see below), 
then $g$ is a \emph{\p-weak upper gradient} of~$u$. 
\end{deff}

We say that a property holds for \emph{\p-almost every curve}
if it fails only for a curve family $\Ga$ with \emph{zero \p-modulus}, 
i.e.\ there is a Borel function $0\le\rho\in L^p(X)$ such that 
$\int_\ga \rho\,ds=\infty$ for every curve $\ga\in\Ga$.
The \p-weak upper gradients were introduced in
Koskela--MacManus~\cite{KoMc}. It was also shown therein
that if $g \in \Lploc(X)$ is a \p-weak upper gradient of $u$,
then one can find a sequence $\{g_j\}_{j=1}^\infty$
of upper gradients of $u$ such that $\|g_j-g\|_{L^p(X)} \to 0$.

If $u$ has an upper gradient in $\Lploc(X)$, then
it has a \emph{minimal \p-weak upper gradient} $g_u \in \Lploc(X)$
in the sense that for every \p-weak upper gradient $g \in \Lploc(X)$ of $u$ we have
$g_u \le g$ a.e., see Shan\-mu\-ga\-lin\-gam~\cite{Sh-harm}.
The minimal \p-weak upper gradient is well defined
up to a set of measure zero.

Following Shanmugalingam~\cite{Sh-rev}, 
the \emph{Newtonian space} 
$\Np(X)=\Np(X,\mu)$ is
the collection of all measurable functions 
$u:X\to [-\infty,\infty]$, having an upper gradient in $L^p(X)$,
such that
\[
\|u\|_{\Np(X)} := \biggl( \int_X |u|^p \, d\mu 
               + \int_X g_u^p \, d\mu \biggr)^{1/p} <\infty.
\]

The space $\Np(X)/{\sim}$, where  $u \sim v$ if and only if $\|u-v\|_{\Np(X)}=0$,
is a Banach space and a lattice, see~\cite{Sh-rev}.
Contrary to the usual a.e.-defined Sobolev functions, 
the functions in $\Np(X)$ are defined everywhere (with values in 
$[-\infty,\infty]$), and
$u \sim v$ if and only if $u=v$ outside a set of \p-capacity zero,
which is important in Proposition~\ref{prop-ac}.

In this paper, the  letter $C$ will denote various positive
constants whose values may vary even within a line.

\section{Doubling and Poincar\'e inequalities}
\label{sect-doubl-PI}

We will discuss several notions of locally \p-admissible measures
on the real line, relations between them, and connections to 
local and global
Muckenhoupt $A_p$ weights. 
We give the following definitions.
Let $B_0=B(x_0,r_0)$.

The measure \emph{$\mu$ is doubling within $B_0$}
if there is $C>0$ (depending on $x_0$ and $r_0$)
such that 
\begin{equation}  \label{eq-deff-doubl}
\mu(2B)\le C \mu(B)
\end{equation}
for all balls $B \subset B_0$.

The
\emph{\p-Poincar\'e inequality holds within $B_0$} 
if there are constants $C>0$ and $\lambda \ge 1$ (depending on $x_0$ and $r_0$)
such that for all balls $B\subset B_0$, 
all integrable functions $u$ on $\la B$, and all upper gradients $g$ of $u$
(or equivalently all \p-weak upper gradients $g$ of $u$), 
\begin{equation}    \label{eq-def-PI-B}
 \vint_{B} |u-u_B| \,d\mu
       \le C r_B \biggl( \vint_{\la B} g^{p} \,d\mu \biggr)^{1/p},
\end{equation}
where $u_B=\vint_B u \,d\mu = \int_B u \,d\mu/\mu(B)$.
We also say that the 
\emph{Lipschitz \p-Poincar\'e inequality holds within $B_0$} 
if \eqref{eq-def-PI-B} holds for all
Lipschitz functions $u$ on $\la B$ with 
\[
   g(x)=\Lip u(x)=\limsup_{y \to x} \frac{|u(y)-u(x)|}{d(y,x)}.
\]

The measure $\mu$ is \emph{\textup{(}Lipschitz\/\textup{)} \p-admissible within $B_0$}
if it is doubling  and supports a (Lipschitz) \p-Poincar\'e
inequality within $B_0$.
Moreover, $w$ is an \emph{$A_p$ weight within $B_0$}
if
\begin{equation} \label{eq-Ap-def}
\vint_B w\, dx < C 
\begin{cases} 
       \biggl( \displaystyle\vint_{B} w^{1/(1-p)}\,dx \biggr)^{1-p},
         &1<p<\infty, \\
      \displaystyle \essinf_B w, &p=1,  \end{cases}  
\end{equation}
for all balls $B \subset B_0$.

Note that whether a property holds within a ball or not depends also 
on the ambient space $X$, since $2B\setm B_0$ or $\la B\setm B_0$ may be nonempty.
Unless said otherwise, the ambient space will be assumed to be $\R$
in this paper.

Furthermore, a property such as these above
is \emph{local} 
if for every $x \in X$ there 
are $R_x,C_x>0$ and $\la_x\ge 1$ such that it holds within
the ball $B(x,R_x)$ with constants $C_x$ and $\la_x$.
If it holds within every ball $B_0\subset X$, with $C$ and $\la$ depending on $B_0$,
it is  \emph{semilocal}.
If, moreover, the constants $C$ and $\la$ are  independent of $B_0$ 
then the property is \emph{global}.
If $d\mu=w\,dx$,
we will sometimes say that $w$ has a property if $\mu$ has it.

As $\Lip u$ is an upper gradient of $u$,
the Lipschitz \p-Poincar\'e inequality trivially follows
from the standard \p-Poincar\'e inequality \eqref{eq-def-PI-B}
within any ball.
If $\mu$ is globally doubling on  a complete metric space, then $\mu$
supports a global Lipschitz \p-Poincar\'e inequality if and only if
it supports the global standard \p-Poincar\'e inequality
\eqref{eq-def-PI-B},
by Keith~\cite[Theorem~2]{Keith} (or 
\cite[Theorem~8.4.2]{HKST}).
As we shall see, the corresponding equivalence is true
also in the local (and semilocal) case on $\R$.

The \p-Poincar\'e inequality (20.3) in 
Heinonen--Kilpel\"ainen--Martio~\cite[2nd ed.]{HeKiMa}
is weaker than the standard \p-Poincar\'e inequality
\eqref{eq-def-PI-B}
but stronger than the Lipschitz \p-Poincar\'e inequality
defined above. 
On $\R$, in view of 
\cite[Theorem~2]{Keith} (or 
 \cite[Theorem~8.4.2]{HKST})
and Theorem~\ref{thm-p-adm-char} below,
the \p-Poincar\'e inequality (20.3) in~\cite[2nd ed.]{HeKiMa}
is  equivalent to 
both  \p-Poincar\'e inequalities considered in this paper,
provided that $\mu$ is locally doubling.

It was shown in Haj\l asz--Koskela~\cite{HaKo-meets} that in
geodesic spaces with a doubling measure supporting a \p-Poincar\'e inequality, 
the dilation constant $\la$ in~\eqref{eq-def-PI-B} can be taken to be 1. 
This is true also under the local assumptions considered here
both for the standard and the Lipschitz \p-Poincar\'e inequality, cf.\ 
Bj\"orn--Bj\"orn~\cite[Theorem~5.1]{BBsemilocal}.
In the rest of this paper, we will 
mainly be concerned with the real line $\R$, and
in this case this can be deduced much more simply and generally.
In particular, the doubling assumption is not required.

\begin{prop} \label{prop-la=1}
Let $\mu$ be a measure on $\R$.
Assume that $I \subset \R$ is a bounded open interval such that
the \p-Poincar\'e inequality \eqref{eq-def-PI-B} holds for $B=I$
with dilation constant  $\la \ge 1$ and constant $\CPI$.
Then \eqref{eq-def-PI-B} also holds for $B=I$   with dilation 
constant $1$ 
and the same constant $\CPI$.
\end{prop}

The proof can be easily modified for Lipschitz Poincar\'e inequalities,
and also for so-called $(q,p)$-Poincar\'e inequalities.
We assume implicitly, as always in this paper, that balls have
finite and positive measure.

\begin{proof}
We may assume that $I=(-1,1)$.
Let $u$ be a bounded integrable function on $I$ and $g$ be an upper gradient
of $u$ on $I$.
For $0 < \eps<1$, we let
\[
    u_\eps(x)=\begin{cases}
         u(x), & \text{if } |x| \le 1-\eps, \\
         u(\pm(1-\eps)), & \text{if } \pm x \ge 1-\eps, 
      \end{cases}
      \quad \text{and} \quad
    g_\eps(x)=\begin{cases}
         g(x), & \text{if } |x| \le 1-\eps, \\
         0, & \text{if } |x| > 1-\eps.
      \end{cases}
\]
It is easy to see that $g_\eps$ is an upper gradient of $u_\eps$.
By dominated convergence, as $u$ is bounded, 
\begin{align*}
 \vint_{I} |u-u_I| \,d\mu 
    & = \lim_{\eps \to 0} \vint_{I} |u_\eps-(u_\eps)_I| \,d\mu  
 \le \lim_{\eps \to 0}\CPI  \biggl(\vint_{\la I} g_\eps^{p} \,d\mu \biggr)^{1/p}
  \\ &
       \le \lim_{\eps \to 0}\CPI  \biggl(\vint_{I} g_\eps^{p} \,d\mu \biggr)^{1/p}
       \le \CPI  \biggl(\vint_{I} g^{p} \,d\mu \biggr)^{1/p}.
\end{align*}
By \cite[Proposition~4.13]{BBbook}, the \p-Poincar\'e
inequality holds also for unbounded integrable $u$.
\end{proof}

\section{Proof of Theorem~\ref{thm-p-adm-char-intro}}

The aim of this section is to show the following
characterization of (semi)locally \p-admissible measures by means of
$A_p$ weights.
The corresponding global characterization was given in 
Bj\"orn--Buckley--Keith~\cite[Theorem~2]{BBK-Ap}.

\begin{thm} \label{thm-p-adm-char}
Let $\mu$ be a measure on $\R$.
Then the following are equivalent\/\textup{:}
\begin{enumerate}
\item \label{c-Lip-loc}
$\mu$ is locally Lipschitz \p-admissible\/\textup{;}
\item \label{c-loc}
$\mu$ is locally \p-admissible\/\textup{;}
\item \label{c-semi}
$\mu$ is semilocally \p-admissible\/\textup{;}
\item \label{c-loc-Ap}
$d\mu=w\,dx$, where $w$ is a local $A_p$ weight\/\textup{;}
\item \label{c-semi-Ap}
$d\mu=w\,dx$, where $w$ is a semilocal $A_p$ weight\/\textup{;}
\item \label{c-glob-Ap}
$d\mu=w\,dx$, and for each bounded interval  $I\subset \R$ there is 
a global $A_p$ weight $\wt$ on $\R$ such that $\wt=w$ on $I$.
\end{enumerate}
\end{thm}

Our first goal is to justify the implications 
\ref{c-semi-Ap} $\imp$ \ref{c-glob-Ap} 
and  \ref{c-loc-Ap} $\imp$ \ref{c-loc}.
This will be implied by the following lemma and its corollary.

\begin{lem} \label{lem-Ap-refl}
Assume that $w$
satisfies the $A_p$ condition~\eqref{eq-Ap-def} 
within the interval $I_0=(0,M)$.
Then the periodically reflected weight
\[
\wh(x) = \begin{cases}
       w(2kM-x), & x\in[(2k-1)M,2kM], \\
       w(x-2kM), & x\in[2kM,(2k+1)M],
     \end{cases}   \quad k \in \Z,
\]
is a global $A_p$ weight on $\R$.
\end{lem}

For weights on $\R^n$, similar reflection results were obtained in 
Bj\"orn~\cite[Proposition~10.5]{ABBergman}
and Rychkov~\cite[Lemma~1.1]{Ryckhov}, but since the proof on $\R$
becomes especially simple, we provide it here for the 
reader's convenience.

\begin{proof}
Let $I\subset\R$ be a bounded open interval. 
If $|I|\le M$ then $I$ intersects
at most two copies $I_k:=(kM,(k+1)M)$ and $I_{k+1}$ of $I_0$.
We can assume that $k=0$ and that $|I\cap I_0| \ge |I\cap I_{1}|$.
Using the $A_p$ condition~\eqref{eq-Ap-def} for $w$ on $B=I\cap I_0$,
we see that for $p>1$,
\begin{align*}
\int_I \wh\,dx \biggl( \int_I \wh^{1/(1-p)}\,dx \biggr)^{p-1} 
&\le 2\int_{I\cap I_0} w\,dx \biggl( 2 \int_{I\cap I_0} w^{1/(1-p)}\,dx \biggr)^{p-1} \\
&\le 2^p C |I\cap I_0|^p \le 2^p C |I|^p.
\end{align*}
On the other hand, if $|I|>M$, then by translating 
$I$,
we can assume that 
$I \subset (0,nM)$ with $nM < 3 |I|$.
Then
\begin{align*}
\int_I \wh\,dx \biggl( \int_I \wh^{1/(1-p)}\,dx \biggr)^{p-1} 
&\le n\int_{I_0} w\,dx \biggl( n \int_{I_0} w^{1/(1-p)}\,dx \biggr)^{p-1} \\
&\le n^p C |I_0|^p \le C |I|^p.
\end{align*}
After division by $|I|^p$,  we obtain~\eqref{eq-Ap-def} 
for $\wh$ on both types of $I$.

For $p=1$, the proof is similar using that
$\essinf_I w=\essinf_{I\cap I_0} w$.
\end{proof}

\begin{cor}  \label{cor-loc-Ap-imp-loc-adm}
If $d\mu=w\,dx$ satisfies the $A_p$ condition~\eqref{eq-Ap-def} within
a bounded
 open interval $I_0$ then $\mu$ is \p-admissible within $\tfrac12 I_0$.
\end{cor}

For the converse implication see Theorem~\ref{thm-local-Ap}.
In view of Proposition~\ref{prop-la=1},
we assume that the dilation constant $\la=1$ from now on.

\begin{proof}
The extension $\wh$ of $w$, provided by Lemma~\ref{lem-Ap-refl},
is a global $A_p$ weight and thus globally \p-admissible on $\R$,
by Theorem~15.21 in Heinonen--Kilpel\"ainen--Martio~\cite{HeKiMa}
(for $p>1$) and Theorem~4 in Bj\"orn~\cite{JBFennAnn} (for $p \ge 1$).
It then follows that for every interval $I\subset \tfrac12 I_0$,
we have $2I\subset I_0$ and thus 
\[
\mu(2I)=\muh(2I)\le C\muh(I)= C\mu(I),
\] 
by the doubling condition for $d\muh=\wh\,dx$ within $I_0$. 
Moreover, the \p-Poincar\'e inequality~\eqref{eq-def-PI-B}
holds for $\wh$ (and thus for $w$) on $I$.
\end{proof}

Most of the rest of this section is
devoted to showing  that \ref{c-Lip-loc} $\imp$ \ref{c-loc-Ap},
and simultaneously that \ref{c-semi} $\imp$ \ref{c-semi-Ap},
in Theorem~\ref{thm-p-adm-char}.

Globally \p-admissible measures on $\R$ are known to be global $A_p$
weights and, in particular, absolutely continuous with respect to
the Lebesgue measure, cf.\ Bj\"orn--Buckley--Keith~\cite{BBK-Ap}.
Next, we obtain a similar characterization for locally 
Lipschitz \p-admissible measures. 
This will be done using reflections and a flattening argument from 
Durand-Cartagena--Li~\cite{Esti-Xining}.
Verifying \p-admissibility for reflected measures
turns out to be more
involved than for the $A_p$ condition in Lemma~\ref{lem-Ap-refl}.

\begin{lem}  \label{lem-adm-on-circle}
Assume that $\mu$ is Lipschitz
\p-admissible within the interval $(-M,2M)\subset\R$
and define the measure $\muh$ on $[-M,M]$ by
\[
\muh(A) = \mu(A\cap[0,M]) + \mu(-A\cap[0,M]),
\quad \text{where} \quad -A=\{x\in\R:-x\in A\}.
\]
Let the metric space $(X,d,\muh)$ be obtained by identifying the endpoints
$\pm M$ of the interval $[-M,M]$ with each other 
and inheriting the length metric from the circle of radius $M/\pi$.
Then $\muh$ is doubling and supports a \p-Poincar\'e inequality on $X$.
\end{lem}

Because of the local doubling property, $\mu$ and thus $\muh$ is nonatomic.
Note that the doubling and Poincar\'e constants for $\muh$ depend on
those for $\mu$ within $(-M,2M)$. 
Example~\ref{ex-not-extension} below shows that it is not enough
to assume that $\mu$ is \p-admissible within the interval $(0,M)$,
even though $\muh$ only depends on $\mu|_{[0,M]}$.

Equivalently, $(X,d,\muh)$ can be obtained by letting
\[
\muh(A) = \mu(A\cap[0,M]) + \mu((M-A)\cap[0,M]),
\]
where $A\subset [0,2M]$ and $M-A=\{x\in\R:M-x\in A\}$,
and identifying the points $0$ and $2M$.
Thus, the reflection points $0$ and $M$ play symmetric roles.

\begin{proof}
We begin by proving the doubling condition~\eqref{eq-deff-doubl} and 
the \p-Poincar\'e inequality~\eqref{eq-def-PI-B} for 
$\muh$ and $I=(x-r,x+r)$,
where $x\in [-\tfrac12M,\tfrac12M]$ and $0<r\le \tfrac14M$.
Intervals centred at $x\in X\setm [-\tfrac12M,\tfrac12M]$, and 
of length at most $\tfrac12 M$, can be treated similarly by reflecting at
 $M$.
Since $X$ is compact, the global doubling and \p-Poincar\'e inequality
then follow  for all $I\subset X$,
by Proposition~1.2 and Theorem~1.3 in
Bj\"orn--Bj\"orn~\cite{BBsemilocal}.

By symmetry, we can assume that $0\le x\le 2r$.
(If $2r<x\le\tfrac12 M$ then $2I\subset[0,M]$ and the doubling property
for $\muh$ and $I$  is immediate.)
From the doubling property of $\mu$ within $(-M,2M)$ it follows that
the measures $\mu((0,a))$ and $\mu((-a,0))$ are comparable for every
$0<a<M$. 
Namely, with $C_d$ being the doubling constant within $(-M,2M)$,
\[ 
\mu((0,a)) \le 
\mu((-a,a)) \le C_d \mu\bigl(\bigl(-\tfrac12a,\tfrac12a\bigr)\bigr) 
\le C_d \mu\bigl(\bigl(-\tfrac32a,\tfrac12a\bigr)\bigr) \le C_d^2 \mu((-a,0)), 
\] 
and similarly 
$\mu((-a,0)) \le C_d^2 \mu((0,a))$.
Hence,
\[
\muh(2I) \le 2\mu((0,x+2r)) 
\le 2\mu(2I) 
\le 2C_d \mu(I) \le 4C_d^3\mu(I\cap[0,M)) \le 4C_d^3\muh(I)
\]
and thus $\muh$ is doubling on $X$.

We shall now prove the \p-Poincar\'e inequality for $\muh$.
As above, and by symmetry, we let $I=(x-r,x+r)$ with
$0 \le x\le \frac12 M$  and $r \le \frac14 M$.
Since $X$ is complete, it suffices to verify the Lipschitz 
\p-Poincar\'e inequality on $I$,
cf.\ Keith~\cite[Theorem~2]{Keith} (or 
\cite[Theorem~8.4.2]{HKST}).

If $r\le x$, then $I\subset [0,M]$ and the 
Lipschitz \p-Poincar\'e inequality 
for $\muh$ follows directly from the one for $\mu$.
Assume therefore that $0 \le x <r \le \frac14 M$.
Let $u$ be a Lipschitz function on $I$.
We can also assume that $u(0)=0$ and thus $\uhat:=u\chi_{(0,x+r)}$
is also Lipschitz.
Let $I'=(-r,x+r)$.
Since $\mu((-r,0))$, $\mu(I')$ and $\mu(I)$ are all comparable, 
because of 
the doubling property of $\mu$, 
the proof of Lemma~2.1 in Kinnunen--Shanmugalingam~\cite{KiSh01} shows that
\[
\vint_{I}|\uhat|\,d\muh
\le C \vint_{I'}|\uhat|\,d\mu
\le C |I'| \biggl( \vint_{I'} (\Lip \uhat)^p 
\,d\mu \biggr)^{1/p}
\le C |I| \biggl( \vint_{I} (\Lip u)^p 
d\muh \biggr)^{1/p}.
\]
The integral of $\ut:=u\chi_{(x-r,0)}$ over $I$ is estimated similarly 
using reflection, and hence
\begin{align*}
\vint_{I}|u-u(0)|\,d\muh 
&\le C\biggl(\vint_{I}|\uhat|\,d\muh + \vint_{I}|\ut|\,d\muh\biggr)
\le C |I| \biggl( \vint_{I} (\Lip u)^p 
\,d\muh \biggr)^{1/p}.
\end{align*}
Finally, we note that 
$\vint_I |u-u_{I,\muh}|\,d\muh \le 2\vint_{I}|u-u(0)|\,d\muh$,
see \cite[Lemma~4.17]{BBbook}.
\end{proof}

\begin{cor}  \label{cor-loc-adm-R}
If $\mu$ is Lipschitz
\p-admissible within an open interval $I_0\subset\R$, then 
it is absolutely continuous with respect to the Lebesgue measure on $I_0$.
\end{cor}

\begin{proof}
Let $x\in I_0$ be arbitrary.
By translation, we can assume that $x=0$ and that $(-M,2M)\subset I_0$ 
for some $M>0$.
Lemma~\ref{lem-adm-on-circle} then shows that the metric space $(X,d,\muh)$, 
obtained by reflecting $\mu$ at $0$ and
identifying the points $\pm M$ with each other, supports a \p-Poincar\'e
inequality with $\muh$ doubling.
Now, flattening $(X,d,\muh)$ as in Durand-Cartagena--Li~\cite{Esti-Xining},
we obtain $\R$ with the measure 
\[
d\mut(y) = \frac{d\muh(y)}{\muh((-|y|,|y|))} 
= \frac{d\muh(y)}{2\mu((0,|y|))} \quad \text{for } y\in X,
\]
which, by~\cite[Theorem~4.1]{Esti-Xining}, is $q$-admissible for 
some sufficiently large $q$. 
Theorem~2 in Bj\"orn--Buckley--Keith~\cite{BBK-Ap} then implies that $\mut$ 
must be an $A_q$ weight on $\R$ and, in particular, 
it is absolutely continuous with respect to the Lebesgue measure. 
It follows that also $\muh$ is absolutely continuous with respect to the
Lebesgue measure on $[0,M]$.
Applying this to all $x\in I_0$, with corresponding $M$, shows that $\mu$
is absolutely continuous with respect to the Lebesgue measure on $I_0$.
\end{proof}

\begin{thm}   \label{thm-local-Ap}
Assume that $\mu$ is Lipschitz
\p-admissible within a bounded open interval $I_0$
and let $\theta>1$. 
Then $d\mu=w\,dt$ for some nonnegative weight $w$ and $w$ is
an $A_p$ weight within $\theta^{-1} I_0$, 
with an $A_p$ constant depending on $\theta$ and $\mu$.
\end{thm}

\begin{proof}
Corollary~\ref{cor-loc-adm-R} shows that 
$d\mu=w\,dt$ for some weight
function $w$ on $I_0$.
Let $I=(x-r,x+r)\subset \theta^{-1} I_0$. Then
$\theta I\subset I_0$.

First, we consider the case $p>1$
and test the Lipschitz
\p-Poincar\'e inequality on $\theta I$ with the Lipschitz function
\[
u(y):= \int_{-\infty}^y \frac{w(t) \chi_I(t)}{(w(t)+\eps)^{p/(p-1)}}\,dt,
\]
where $\eps>0$ is fixed but arbitrary.
Note that $u\equiv0$ on the left component $I_L=(x-\theta r,x-r]$ of 
$\theta I\setm I$ and that $u\equiv u(x+r)$ 
on the right component $I_R=[x+r,x+\theta r)$ of $\theta I\setm I$.
Hence, at least one of the following holds
\[
|u-u_{\theta I}| \ge\tfrac12 u(x+r) \text{ on } I_L
\quad \text{or} \quad
|u-u_{\theta I}| \ge\tfrac12 u(x+r) \text{ on } I_R.
\]
Since $\mu(I_L)$ and $\mu(I_R)$ are comparable to $\mu(\theta I)$, 
with comparison constant depending on $\theta$,
this implies that
the left-hand side of the Lipschitz \p-Poincar\'e inequality on $\theta I$ is
\[
\vint_{\theta I} |u-u_{\theta I}|\,d\mu \ge Cu(x+r) 
= C\int_I \frac{d\mu}{(w+\eps)^{p/(p-1)}}.
\]
At the same time, $\Lip u = |u'|\le (w+\eps)^{-1/(p-1)}\chi_I$ a.e. 
(by the fundamental theorem of calculus) and thus
the right-hand side of the Lipschitz \p-Poincar\'e inequality is 
\[
Cr \biggl( \vint_{\theta I} (\Lip u)^p\,d\mu \biggr)^{1/p}
\le Cr \biggl( \frac{1}{\mu(\theta I)} \int_{I} 
           \frac{d\mu}{(w+\eps)^{p/(p-1)}} \biggr)^{1/p}.
\]
Combining
the last 
two estimates with the Lipschitz \p-Poincar\'e inequality yields
\[
\mu(\theta I)^{1/p} 
  \le Cr\biggl( \int_I \frac{d\mu}{(w+\eps)^{p/(p-1)}} \biggr)^{1/p-1}
  = Cr\biggl( \int_I \frac{w(t)\,dt}{(w(t)+\eps)^{p/(p-1)}} \biggr)^{1/p-1}.
\]
Raising the last estimate to the $p$th power, writing 
$\int_I w(t)\,dt=\mu(I)\le\mu(\theta I)$, dividing by $|I|=2r$ 
and letting $\eps\to0$ 
we obtain~\eqref{eq-Ap-def} for $I$, by monotone convergence.

Now, we consider $p=1$ and let $m=\essinf_I w$.
Test the Lipschitz $1$-Poincar\'e inequality on $\theta I$ with the Lipschitz function
\[
u(y):= \int_{-\infty}^y \chi_{E_\eps}(t)\,dt,
\]
where $E_\eps=\{t\in I: w(t)<m +\eps\}$ and
$\eps>0$ is fixed but arbitrary.
Then, as above, $u(x+r)=|E_\eps|$ is majorized by the right-hand side 
in the Lipschitz $1$-Poincar\'e inequality, i.e.\ 
\[
0< |E_\eps| 
\le \frac{Cr}{\mu(\theta I)} \int_{\theta I} \chi_{E_\eps}(t) w(t)\,dt 
\le Cr \frac{(m+\eps)|E_\eps|}{\mu(I)}.
\]
Dividing by $|E_\eps|>0$ and letting $\eps\to0$, 
yields~\eqref{eq-Ap-def} for $I$, i.e.\ within $\theta^{-1} I_0$. 
\end{proof}

\begin{proof}[Proof of Theorem~\ref{thm-p-adm-char}]
\ref{c-semi-Ap} $\imp$ \ref{c-glob-Ap} 
This follows from Lemma~\ref{lem-Ap-refl}.

\ref{c-glob-Ap} $\imp$ \ref{c-semi-Ap} $\imp$ \ref{c-loc-Ap}
and \ref{c-loc} $\imp$ \ref{c-Lip-loc}
These implications are trivial.

\ref{c-loc-Ap} $\imp$  \ref{c-loc}
This follows from 
Corollary~\ref{cor-loc-Ap-imp-loc-adm}.

\ref{c-loc} $\imp$ \ref{c-semi}
This follows from Proposition~1.2 and Theorem~1.3 in 
Bj\"orn--Bj\"orn~\cite{BBsemilocal}.

\ref{c-Lip-loc} $\imp$ \ref{c-loc-Ap} and
\ref{c-semi} $\imp$ \ref{c-semi-Ap}
These implications follow from Theorem~\ref{thm-local-Ap}.
\end{proof}

Lemma~\ref{lem-Ap-refl} and Theorem~\ref{thm-local-Ap} imply that
if $\mu$ is \p-admissible within $(-\theta M,\theta M)$ for some $\theta>1$,
 then the periodically
repeated reflections of $\mu|_{[-M,M]}$ provide a \p-admissible measure on $\R$.
For the above arguments to hold it is important
that \p-admissibility is assumed within a larger interval. 
Next, we give an example showing that
it is not enough to assume \p-admissibility within $(-M,M)$.

\begin{example}   \label{ex-not-extension}
Let $X=[0,1]$, $w(x)=x^\alp$, $\alp \ge 0$, and $d \mu = w\,dx$.
Then $\mu$ is doubling on $X$.
By Chua--Wheeden~\cite[Theorem~1.4]{ChuaWheeden},
$\mu$ supports a $1$-Poincar\'e inequality for the 
interval $(a,b)\subset [0,1]$ with $\la=1$ and the optimal constant
\begin{align*}
     C
 &=\frac{2}{(b-a)\mu(a,b)} 
     \biggl\|\frac{\mu((a,x))\mu((x,b))}{w(x)} \biggr\|_{L^\infty(a,b)} 
    \le \frac{2}{(b-a)} 
     \biggl\|\frac{\mu((a,x))}{w(x)} \biggr\|_{L^\infty(a,b)} \\\
 & = \frac{2}{(b-a)} 
     \biggl\|\frac{x^{1+\alp}-a^{1+\alp}}{(1+\alp)w(x)} \biggr\|_{L^\infty(a,b)} 
  = \frac{2}{(b-a)} 
     \biggl\|\frac{(x-a) \xi_x^\alp}{x^\alp} \biggr\|_{L^\infty(a,b)} 
  \le  2,
\end{align*}
where $\xi_x \in(a,x)$ comes from the mean-value theorem.
As this holds for all intervals $(a,b)\subset [0,1]$, 
$\mu$ supports a $1$-Poincar\'e inequality 
with respect to the metric space
$[0,1]$,
with constant $2$.
It follows that $\mu$, extended by $0$ outside $[0,1]$,
is \p-admissible within $(0,1)$.

If $p < 1+\alp$, then $w^{1/(1-p)}$ is not integrable at $0$, and hence
the $A_p$ condition~\eqref{eq-Ap-def} does not hold for $w$ within $(0,1)$.
It is also easily verified that
the set $\{0\}$ has zero capacity
with respect to the metric space $[0,1]$.
This
implies that the collection of all nonconstant 
compact rectifiable curves in $[0,1]$ starting at $0$ has \p-modulus 
zero (see \cite[Proposition~1.48]{BBbook}).
Hence, 
$\chi_{(0,\infty)}$ has $0$ as a \p-weak upper gradient on $(\R,\muh)$
for any extension $\muh$ of $\mu$ to $\R$, which
violates the \p-Poincar\'e inequality
on the interval $(-1,1)$.
Thus, $\mu$ is \emph{not}
a restriction of any measure on $\R$ supporting a \p-Poincar\'e inequality.
\end{example}

\section{Consequences of Theorem~\ref{thm-p-adm-char-intro}}
\label{sect-cor-Ap}

\begin{cor}
Let $\mu_j$, $j=1,2$, be locally \p-admissible measures on $\R$.
Then $\mu=\mu_1+\mu_2$ is locally \p-admissible on $\R$.
\end{cor}

\begin{cor}
Let $w_j$, $j=1,2$, be locally \p-admissible weights on $\R$.
Then $\max\{w_1,w_2\}$ and $\min\{w_1,w_2\}$ are locally \p-admissible on $\R$.
Moreover, for $p>1$, the weight $w_1^{1/(1-p)}$ is locally $p/(p-1)$-admissible 
on $\R$.
\end{cor}

These statements follow directly from the characterizations in 
Theorem~\ref{thm-p-adm-char-intro} together with similar statements for
global $A_p$ weights. 
The lattice property of global $A_p$ weights on $\R^n$
was proved in 
Kilpel\"ainen--Koskela--Masaoka~\cite[Proposition~4.3]{KilKoMa}
using nontrivial characterizations of $A_p$ and $A_\infty$ weights.
Here we seize the opportunity to provide 
an elementary proof, including $p=1$
and also covering the local case.

It is 
straightforward that the $A_p$ 
condition
\begin{equation} \label{eq-Ap-def-2}
\vint_B w\, dx < C 
\begin{cases} 
       \biggl( \displaystyle\vint_{B} w^{1/(1-p)}\,dx \biggr)^{1-p},
         &1<p<\infty, \\
      \displaystyle \essinf_B w, &p=1,  \end{cases}  
\end{equation}
for $w$ is precisely the $A_{p/(p-1)}$ condition
for the conjugate weight $w^{1/(1-p)}$ with the $A_{p/(p-1)}$ constant
$C^{1/(p-1)}$ when $p>1$.

\begin{lem} \label{lem-Ap-lattice}
Assume that the $A_p$ condition 
holds for $w_1$ and $w_2$ 
with a constant~$C$ in some ball $B\subset\R^n$.
Then it holds also for $w_1+w_2$, $\max\{w_1,w_2\}$ 
and $\min\{w_1,w_2\}$ with constants $2C$, $2C$ and $2^{p-1} C$,
respectively.
\end{lem}

\begin{proof}
We have
\begin{align*}
\vint_B \max\{w_1,w_2\}\,dx &\le \vint_B (w_1+w_2)\,dx
= \vint_B w_1\,dx + \vint_B w_2\,dx, \\
\vint_B \min\{w_1,w_2\}\,dx 
&\le \min \biggl\{ \vint_B w_1\,dx, \vint_B w_2\,dx \biggr\}.
\end{align*}
For $p=1$, \eqref{eq-Ap-def-2} then follows directly from the facts that
\begin{align*}
\essinf_B w_1 + \essinf_B w_2 &\le 2 \essinf_B \max\{w_1,w_2\} 
\le 2 \essinf_B (w_1+w_2), \\ 
\min \Bigl\{ \essinf_B w_1, \essinf_B w_1 \Bigr\} &= \essinf_B \min\{w_1,w_2\}.
\end{align*}
For $p>1$, we have
\begin{align*}
\vint_B w_1\,dx + \vint_B w_2\,dx 
< C \biggl( \vint_{B} w_1^{1/(1-p)}\,dx \biggr)^{1-p} +
C \biggl( \vint_{B} w_2^{1/(1-p)}\,dx \biggr)^{1-p}.
\end{align*}
Since $1-p<0$ and
\begin{align*}
\vint_{B} w_j^{1/(1-p)}\,dx 
\ge \vint_{B} \max\{w_1,w_2\}^{1/(1-p)}\,dx 
\ge \vint_{B} (w_1+w_2)^{1/(1-p)}\,dx 
\end{align*}
for $j=1,2$,
this proves \eqref{eq-Ap-def-2} for $w_1+w_2$ and $\max\{w_1,w_2\}$.

To prove \eqref{eq-Ap-def-2} for $\min\{w_1,w_2\}$, we note that
\[
\min\{w_1,w_2\}^{1/(1-p)} = \max \bigl\{w_1^{1/(1-p)},w_2^{1/(1-p)} \bigr\},
\]
which by the above argument satisfies the $A_{p/(p-1)}$ condition 
with $2C^{1/(p-1)}$.
The duality between \eqref{eq-Ap-def-2} and the $A_{p/(p-1)}$ condition
concludes the proof. 
\end{proof}

\section{Proof of Proposition~\ref{prop-ac}}

\begin{proof}[Proof of Proposition~\ref{prop-ac}]
By considering a smaller interval if necessary, we can assume that 
$I$ is closed,
$u\in\Np(I,\mu)$ and $w,w^{1/(1-p)}\in L^1(I)$.
Let $g\in L^p(I,\mu)$ be an upper gradient of $u$.
Let $\eps>0$ be arbitrary and 
find $\de>0$ so that 
\begin{equation}   \label{eq-abs-cont-meas}
\int_E w^{1/(1-p)}\,dx<\eps \quad \text{whenever } E\subset I
\text{ and }|E|<\de.
\end{equation}
Consider finitely many pairwise disjoint intervals $(a_j,b_j)\subset I$ 
with 
\[
\sum_j|b_j-a_j|<\de. 
\]
The H\"older inequality then yields for each $j$,
\begin{equation}   \label{eq-ub-ua}
|u(b_j)-u(a_j)| \le \int_{a_j}^{b_j} g\,dx
\le \biggl( \int_{a_j}^{b_j} g^p w\,dx \biggr)^{1/p} 
             \biggl( \int_{a_j}^{b_j} w^{1/(1-p)} \,dx \biggr)^{1-1/p}.
\end{equation}
Summing over all $j$ and using~\eqref{eq-abs-cont-meas}
and the H\"older inequality for sums,
we conclude that
\begin{align*}
\sum_{j} |u(b_j)-u(a_j)| 
&\le \biggl( \sum_{j} \int_{a_j}^{b_j} g^p\,d\mu \biggr)^{1/p} 
      \biggl( \sum_{j} \int_{a_j}^{b_j} w^{1/(1-p)} \,dx \biggr)^{1-1/p} \\
&\le \biggl( \int_I g^p\,d\mu \biggr)^{1/p} \eps^{1-1/p}.
\end{align*}
Since $\eps>0$ was arbitrary and $g\in L^p(I,\mu)$, we conclude that $u$
is locally absolutely continuous (and thus a.e.\ differentiable) on $I$.

It remains to show that $|u'|\le g_u$ a.e.,
  since 
the converse inequality is trivial.
Let $x \in \interior I$. 
Replacing $(a_j,b_j)$ in~\eqref{eq-ub-ua} by 
$(x-h,x+h) \subset I$,
 we see that for all upper gradients 
$g\in L^p(I,\mu)$,
\[ 
\frac{|u(x+h)-u(x-h)|}{2h} 
\le \biggl( \frac{1}{2h} \int_{x-h}^{x+h} g^pw\,dx \biggr)^{1/p}
    \biggl( \frac{1}{2h} \int_{x-h}^{x+h} w^{1/(1-p)}\,dx \biggr)^{1-1/p}.
\] 
Letting $h\to0$, together with the observation that a.e.\ $x\in I$ is
a point of differentiability of $u$ as well as a Lebesgue point both of 
$g^pw$ and of $w^{1/(1-p)}$, shows that $|u'|\le g$ a.e.
As this holds for all upper gradients $g$ of $u$, 
and there is a sequence $\{g_j\}_{j=1}^\infty$ of upper gradients
tending to $g_u$ pointwise a.e.,
we conclude that
$|u'|\le g_u$ a.e.,
and thus $|u'| \in L^p(I,\mu)$.
\end{proof}

\section{Uniform assumptions}
\label{sect-uniform}

Sometimes it can be of interest to consider (semi)locally
admissible measures with uniform constants.
We therefore introduce the following notions.

\begin{deff} \label{deff-unif}
Any of the properties considered in 
Section~\ref{sect-doubl-PI}
is \emph{uniformly local} if there
are $R,C>0$ and $\la \ge 1$
such that the property holds within every ball $B_0\subset X$ of radius $R$,
with the same constants $C$ and $\la$.

The property is \emph{semiuniformly local} if 
for every $x$  it holds within some ball $B(x,R_x)$ 
with constants $C$ and $\la$ independent of $x$ and $R_x$.

If it holds within every ball $B_0$ with $C$ and $\la$ depending on the
radius (but not the centre) of $B_0$, then it is 
\emph{uniformly semilocal}.
\end{deff}

Uniformly local $A_p$ weights were studied by
Rychkov~\cite{Ryckhov} under the name ``local $A_p$ weights''
(primarily for the specific radius $R=1$).

A careful analysis of the proofs in this paper shows that the involved
constants depend on each other in a controllable way.
This, in particular, means that the implications
\ref{c-glob-Ap}  $\eqv$ \ref{c-semi-Ap} $\imp$ \ref{c-loc-Ap}  
$\eqv$ \ref{c-loc}
$\eqv$ \ref{c-Lip-loc}
and \ref{c-semi} $\imp$ \ref{c-semi-Ap} 
in Theorem~\ref{thm-p-adm-char} hold also if the
(semi)local notions are replaced by uniformly (semi)local ones, 
and \ref{c-glob-Ap} is replaced by its uniform version, where the global
$A_p$ constant of the extension $\wt$ 
may depend on $r_B$, but not
on the centre of $B$.

Moreover, 
the covers used in the proofs of 
\cite[Proposition~1.2 and Theorem~1.3]{BBsemilocal} 
(leading to the implication \ref{c-loc} $\imp$ \ref{c-semi}
in Theorem~\ref{thm-p-adm-char})
can be controlled by constants
which only depend on  $C$, $\la$ and the involved radii, but not on $x$.
Since the other estimates therein are quantitative as well, 
also the implication \ref{c-loc} $\imp$ \ref{c-semi} 
in Theorem~\ref{thm-p-adm-char}
holds for  uniformly (semi)locally \p-admissible measures on $\R$.

Note that the ``uniform'' properties require uniformity both in $C$ and $R$,
while the semiuniformity allows $R_x$ to depend on $x$.
In Bj\"orn--Bj\"orn~\cite[Section~6]{BBsemilocal},
this property was  shown 
to be sufficient for many qualitative, as well as some quantitative, 
properties of \p-harmonic functions, but it is not strong enough for the 
uniform conclusions above.
In fact, any positive continuous weight on $\R$ is semiuniformly locally
\p-admissible, but the weight $e^{e^{|x|}}$ is not even 
uniformly locally doubling.

\end{document}